\documentclass[9pt]{jdg-p-1-2}

\newtheorem{theorem}{Theorem}[section]

\newtheorem{corollary}[theorem]{Corollary}

\newtheorem{proposition}[theorem]{Proposition}

\theoremstyle{definition}

\newtheorem{definition}[theorem]{Definition}

\theoremstyle{remark}

\def\Ric{\text{Ric}}

\def\e{\epsilon}

\def\Zbar{{\overline{Z}}}

\def\Wbar{{\overline{W}}}

\def\C{\Bbb C}

\def\R{\Bbb R}

\def\P{\Bbb P}

\def\id{\operatorname{id}}

\def\Ric{\operatorname{Ric}}

\def\I{\operatorname{I}}

\numberwithin{equation}{section}

\begin{document}

\title[Positive Complex Sectional Curvature]{Positive Complex Sectional Curvature,
Ricci Flow and the Differential Sphere Theorem}





\author{Lei Ni}

\address{Department of Mathematics, University of California at San Diego, La Jolla, CA 92093}


\email{lni@math.ucsd.edu}


\author{Jon Wolfson}
\address{Department of Mathematics, Michigan State University, East Lansing, MI 48824}
\email{wolfson@math.msu.edu}

\thanks{The first author was supported in part by NSF grant DMS-0504792  and an Alfred P. Sloan
Fellowship, USA. The second author was supported in part by NSF grant DMS-0604759 }



\date{April 2007}

\date{May 2007}

\keywords{}

\maketitle

\section{ \bf Introduction}

A fundamental question in
 Riemannian geometry is to understand what natural curvature conditions are
  preserved on a Riemannian manifold $M$ under Ricci flow. The well known conditions are:
   positive scalar curvature; positive curvature operator; 2-positive curvature operator;
   positive bisectional curvature and positive Ricci curvature in three
   dimensions. It is also known that the nonnegative sectional
   curvature \cite{N} and nonnegative Ricci curvature \cite{Knopf} are not preserved in general
   (see also \cite{BW1}).
   Recently, Brendle and Schoen \cite{BS} in an important paper
   have added another important condition to this list: positive isotropic
   curvature. See also \cite{Ng}.
    This
   result suggests a program for the classification of compact manifolds with positive isotropic curvature.
    The fact that these various conditions are preserved yields,
   after further study, fundamental results in Riemannian geometry.  For example, B\"ohm and Wilking \cite{BW},
   in foundational work, exploited the preservation of the positivity and
   the 2-positivity of the curvature operator to prove that compact manifolds satisfying either
   one of these conditions are spherical  space forms. Brendle and Schoen used
  the preservation of positive isotropic curvature and an associated curvature condition on
  $M \times \R^2$ to prove that compact manifolds with (pointwise)
  $1/4$-pinched positive sectional curvature admit metrics of constant positive
 curvature and are therefore diffeomorphic to spherical  space forms.
 This resolves the long standing open conjecture, the  differential sphere conjecture.
  For a detailed discussion of the history of this problem see the introduction of
 \cite{BS}.

Another natural curvature condition is  positive complex sectional
curvature, which lies between the positivity of the isotropic
curvature and the positivity of the curvature operator.
 To define this condition on a Riemannian manifold $(M, g)$ consider the complexified tangent bundle
 $TM \otimes \C$. Extend the metric to be symmetric and linear over $\C$ on $TM \otimes \C$ (not hermitian)
 and extend the curvature $R$ linearly over $\C$. Then $(M, g)$ has positive (non-negative) complex
 sectional curvature if for every $p \in M$ and every linearly independent pair of vectors
 $Z, W \in T_p M \otimes \C$:
$$ \langle R(Z, W) \bar{Z}, \bar{W} \rangle > 0, ( \langle R(Z, W) \bar{Z}, \bar{W} \rangle \geq 0).$$

In this short note, we shall show first that the positivity and the
nonnegativity of the complex sectional curvature is preserved under
Ricci flow. Using this, together with techniques of B\"ohm and Wilking, one
can conclude that  if $(M, g)$ is a compact Riemannian manifold with
positive complex sectional curvature then the normalized Ricci flow
deforms the metric to a metric of constant positive curvature. 
We then use earlier work of Yau and Zheng  \cite{YZ} to show that a
metric with strictly (pointwise) $1/4$-pinched sectional curvature
has positive complex sectional curvature. This gives a direct 
proof of Brendle-Schoen's recent differential sphere theorem, bypassing
any discussion of positive isotropic curvature.
A further application of our approach is a characterization of
space forms using a weaker point-wise $1/4$-pinched condition.

\begin{theorem}\label{general}
Let $(M, g_0)$ be a compact Riemannian manifold. Assume that there
exists continuous function $k(p), \delta(p) \ge 0$ such that
$\mathcal{P}=\{p\, | k(p)>0\}$ is dense and $\delta(p_0)>0$ for some
$p_0\in \mathcal{P}$, with the property that
$$
\frac{1+\delta(p)}{4}k(p)\le R(X, Y, X, Y)/|X\wedge Y|^2\le
(1-\delta(p))k(p).
$$
Then the normalized Ricci flow deforms $(M, g_0)$ into a metric of
constant curvature.  Consequently $M$ is diffeomorphic
 to a spherical space form.
\end{theorem}

It turns out that $M$ has  nonnegative complex sectional curvature
is the same as $M\times \R^2$ has nonnegative isotropic curvature.
We discuss this in the last section.

\bigskip

\section{\bf Ricci flow preserves positive complex sectional curvature}

Let $(M, g)$ be a compact Riemannian $n$-manifold.
Recall that the curvature tensor $R$ is said to have nonnegative complex
sectional curvature if  $\langle R(Z, W) \overline{Z}, \overline{W} \rangle \ge 0$
for any $p\in M$ and  $Z, W \in T_pM \otimes \C$, where $\langle -, - \rangle$
denotes the symmetric inner product. We will henceforth use the notation
$$R(Z, W, \overline{Z}, \overline{W})   =\langle R(Z, W) \overline{Z}, \overline{W} \rangle$$
Let $Z=X+\sqrt{-1} Y$ and $W=U+\sqrt{-1}V$ with $X, Y, U, V \in T_pM$.

Choose $Z$ and $W$  so that they satisfy $\nabla Z=0$, $\nabla W=0$,
$D_t Z=0$ and $D_t W=0$, then, using \cite{Hharnack} (see also
\cite{H86}), we have with respect to the moving frame,

\begin{eqnarray*}
\left(D_t-\Delta\right) R(Z, W, \overline{Z}, \overline{W}) &=&
R_{ZWpq}R_{\Zbar\Wbar pq}+2R_{Zp\Zbar q}R_{Wp\Wbar q}-2R_{Zp\Wbar
q}R_{Wp\Zbar q}
\end{eqnarray*}
Here $R_{ZWpq}=R(Z, W, e_p, e_q)$, $p,q = 1, \dots, n$, where
$\{e_p\}$, is a orthonormal frame (also a unitary frame of $T_pM
\otimes \C$). Following the notation of \cite{H86} and \cite{BW} we
denote the right hand side above by $Q(R)(Z, W, \overline{Z},
\overline{W})$ (abbreviated as $Q(R)$).

We will show that  if  $t_0$ is the first time for which there are vectors $Z,W \in T_pM \otimes \C$ such that,
\begin{equation}
\label{eqn:criticalpt} R(Z, W, \overline{Z}, \overline {W})=0.
\end{equation}
then,
$$
Q(R)(Z, W, \overline{Z}, \overline{W})\ge 0.
$$
We remark that from this fact and the Hamilton maximum principle it
follows that Ricci flow preserves positive complex sectional
curvature.

Following \cite{Mok}, for any complex tangent vectors $Z_1$ and
$W_1$ and any real number $s$ define,
$$
f(s)\doteqdot R(Z+sZ_1, W+sW_1, \overline{Z}+s\overline{Z_1},
\overline{W}+s\overline{W_1}) \ge 0
$$
Using that
$$
R(Z, W, \overline{Z}, \overline{W})=0,
$$
it follows that $f(0)=0$ and $f''(0)\ge 0$.
This then implies
\begin{eqnarray}\label{2ndvar1}
0&\le &R(Z_1, W_1, \Zbar, \Wbar)+R(Z_1, W, \Zbar_1, \Wbar)+R(Z_1, W, \Zbar, \Wbar_1)\\
&\, &+R(Z, W_1, \Zbar_1, \Wbar)+R(Z, W_1, \Zbar, \Wbar_1)+R(Z, W,
\Zbar_1, \Wbar_1). \nonumber
\end{eqnarray}
Replacing $Z_1$ by $\sqrt{-1}Z_1$, and $W_1$ by $\sqrt{-1}W_1$
we have that
\begin{eqnarray}\label{2ndvar2}
0&\le &-R(Z_1, W_1, \Zbar, \Wbar)+R(Z_1, W, \Zbar_1, \Wbar)+R(Z_1, W, \Zbar, \Wbar_1)\\
&\, &+R(Z, W_1,\Zbar_1, \Wbar)+R(Z, W_1, \Zbar, \Wbar_1)-R(Z, W,
\Zbar_1, \Wbar_1). \nonumber
\end{eqnarray}
Adding we have,
\begin{eqnarray}\label{2ndvar3}
0&\le &R(Z_1, W, \Zbar_1, \Wbar)+R(Z, W_1, \Zbar,
\Wbar_1)+2\mathcal{R}e\left(R(Z, W_1, \Zbar_1, \Wbar)\right).
\end{eqnarray}

The result now follows from a result of \cite{Mok}. See for example
Lemma 2.86 of \cite{Chowetc} (see also pages 11-12 of \cite{H93}).

We have proved:

\begin{theorem}
\label{thm:PCSC}
The Ricci flow on a compact manifold preserves the cones
consisting of: (i) the curvature operators with nonnegative complex
sectional curvature and (ii) the curvature operators with positive complex
sectional curvature.
\end{theorem}

\bigskip

In \cite{BW}, the following concept is introduced.

\begin{definition}\label{bw-def}
A continuous family $C(s)_{s\in [0,\infty)}$ of closed convex
$O(n)$-invariant cones of full dimension (in the space of algebraic
curvature operators) is called a pinching family if

(1) each $R\in C(s)\setminus \{0\}$ has positive scalar curvature,

(2) $R^2+R^{\#}$ is contained in the interior of the tangent cone of
$C(s)$ at $R$ for all $R\in C(s)\setminus \{0\}$ and all $s\in(0,
\infty).$

(3) $C(s)$ converges in the pointed Hausdorff topology to the
one-dimensional cone $\R_+ \I$ as $s\to \infty$.
\end{definition}

The proof in \cite{BW} (see also the argument of
\cite{BS}) yields the the following theorem.

\begin{theorem} If $C(0)$ is a $O(n)$-invariant cone preserved under
Ricci flow. Assume further that the cone of positive curvature
operators is contained in the interior of $C(0)$ and every $R\in
C(0)$ has nonnegative Ricci curvature. Then there exists a
continuous pinching family $C(s)$ with $C(s)=C(0)$.
\end{theorem}

An immediate corollary is:

\begin{corollary}
If $(M, g)$ is a compact Riemannian manifold with
positive complex sectional curvature, then the normalized Ricci flow
deforms $(M, g)$ to a Riemannian manifold of constant positive curvature.
\end{corollary}

\begin{proof} By Theorem \ref{thm:PCSC} we have that the cone
consisting of the curvature operators with nonnegative complex
sectional curvature is invariant under the Ricci flow. Note that the
cone of positive curvature operators is contained in the interior of
the above cone. Also it is clear that if the curvature operator $R$
has nonnegative complex sectional curvature it must has nonnegative
Ricci curvature. The result now follows from the pinching family
construction of \cite{BW} (see also \cite{BS}) and Theorem 5.1 of \cite{BW}.
\end{proof}

\section{\bf $1/4$-pinched implies positive complex sectional curvature}

Yau and Zheng [YZ] prove that if the sectional curvatures are negative and $1/4$ pinched, that is,
$$
-1\le R(X, Y, X, Y)/|X\wedge Y|^2\le -\frac{1}{4}
$$
for any linearly independent vectors $X, Y \in TM$
then the complex sectional curvature is non-positive, that is, $R(Z, W, \Zbar, \Wbar)\le 0$,
for any linearly independent vectors $Z, W \in TM \otimes \C$.
A slight modification of this argument can be used  to show that
if at $p\in M$, for any $X, Y \in T_pM$,
$$
\frac{1+\delta}{4}k(p)\le R(X, Y, X, Y)/|X\wedge Y|^2\le (1-\delta)k(p),
$$
for some $k(p)>0$ then for any any linearly independent vectors $Z, W \in TM \otimes \C$,  $R(Z, W, \Zbar,
\Wbar)>0$. For the sake of the completion we include the argument here.

\bigskip

We start with the lemma of Berger.

\begin{proposition}[Berger]\label{berger}
Suppose that for any $X, Y \in T_pM$ and $k(p) >0$
$$
\frac{1+\delta}{4}k(p)\le R(X, Y, X, Y)/|X \wedge Y|^2\le (1-\delta)k(p),
$$
 If $\{X, Y, U, V\}\in
T_pM$ are linearly independent and
$$
\Delta=\langle X, U\rangle \langle Y, V\rangle -\langle X,
V\rangle\langle Y, U\rangle =0.
$$
Then
\begin{eqnarray}
6\left|R(X, Y, U, V)\right| &\le& \frac{3-5\delta}{5-3\delta}\big(2R(X,Y)+2R(U,V)\\  \nonumber
&\,& + R(X,V)+R(V,Y)+R(X,U)+R(U, Y)\big).
\end{eqnarray}
Here $R(X, Y)=R(X, Y, X, Y)$.
\end{proposition}

\begin{proof} See \cite{YZ}, proof of Lemma 1. One can let $a=
c=1$ to make the argument more transparent.
\end{proof}

\bigskip

\begin{proposition}\label{dualyz} If the sectional
curvature is pointwise $1/4$ pinched, in the sense that
$$
\frac{k(p)}{4} \le R(X, Y, X, Y)/|X\wedge Y|^2\le k(p),
$$
for some function $k(p)>0$, then $R(Z, W, \Zbar, \Wbar)\ge 0$. Moreover if
$$
\frac{1+\delta}{4}k(p)\le R(X, Y, X, Y)/|X\wedge Y|^2\le
(1-\delta)k(p),
$$
for some $\delta>0$, then there exists $\epsilon>0$
such that $(R-\epsilon \I)(Z, W, \Zbar, \Wbar)\ge 0$, where $\I$ is
the identity (complex extension) of $S^2(\wedge^2(\R^n))$.
\end{proposition}

\begin{proof} We follow the proof of \cite{YZ}. Define the function
$$
f(Z, W)=\frac{R(Z, W, \Zbar, \Wbar)}{|Z|^2|W|^2}
$$
on $(\C^n)^*\times (\C^n)^*$ where $|Z|^2=\langle Z, \Zbar \rangle$.
Clearly $f$ is defined on $\P^{n-1}\times \P^{n-1}$. To prove the first statement
of the Proposition it suffices to show that $f(Z,W) \geq 0$ for all $Z,W \in T_pM \otimes \C$.
We shall prove the result by contradiction. Assume that there exist
a pair of vector $(Z, W)$ such that $f(Z, W)<0$.
Notice that $R(aZ+bW, cZ+dW, \overline{aZ+bW},
 \overline{cZ+dW})=|ad-bc|^2R(Z, W, \Zbar, \Wbar)$. Hence by
replacing $(Z, W)$ by $(Z, W-\frac{\langle W, Z\rangle}{\langle Z,
Z\rangle}Z)$ in the case $\langle Z, Z\rangle \ne0$, or by $(Z-W,
Z+W)$ in the case that both $\langle Z, Z\rangle =\langle W, W\rangle=0$
we can assume $\langle Z, W\rangle=0$ without changing the sign of $f(Z,W)$
(though its absolute value is changed). Thus we can assume that the
minimum of $f(Z, W)$ under the constraint
$$
\langle Z, W\rangle=0.
$$
is achieved and is negative. Suppose the minimum is achieved at $(Z,W)$.
Clearly $Z, W$ are linearly independent. Introduce the Lagrange multiplier,
$$
F(Z, W) = f(Z,W) + \lambda |\langle Z, W\rangle|^2,
$$
Then the minimum point  $(Z,W)$ is a critical point of $F$ and hence,
$$
R(Z, W)\Zbar +f(Z, W)|Z|^2 W=0.
$$
Therefore,
$$
\langle \Zbar, W \rangle =\frac{R(Z, W, \Zbar, \Zbar)}{|Z|^2 f(Z, W)}=0.
$$
Thus $\langle Z, W\rangle=0$ and $\langle \Zbar, W \rangle = 0$. Writing
$Z=X+\sqrt{-1}Y$ and $W=U+\sqrt{-1}V$ we conclude that $\{X,
Y\} \perp \{U, V\}$. Observing that $f(\lambda Z, \mu W)=f(Z, W)$, for any
complex scalars $\lambda$ and $\mu$ we see that we can  adjust
$Z$ and $W$ so that $X\perp Y$ and $U\perp V$. Without loss of
the generality we may assume that $1=|X|\ge |Y|$ and $1=|U|\ge |V|$.
Therefore we have,
\begin{eqnarray*}
R(Z, W, \bar{Z}, \bar{W}) &=&R(X, U, X, U)+R(X, V, X, V)+R(Y, U, Y,
U)+R(Y, V, Y, V)\\
&\quad& -2R(X, Y, U, V)\\
&\ge & R(X, U)+R(X, V)+R(Y, U)+R(Y, V)\\
&\,& -\frac{1}{5}\left(2R(X,Y)+2R(U,V)+R(X,V)+R(V,Y)\right.\\
 &\, &\left.+R(X, U)+R(U, Y)\right)\\
&=& \frac{4}{5} \left(R(X, U)+R(X, V)+R(Y, U)+R(Y, V)\right)\\
&\, &-\frac{2}{5}\left(R(X, Y)+R(U, V)\right)\\
&\ge & \frac{k} {5}\left(1-|V|^2\right)\left(1-|Y|^2\right)\ge 0.
\end{eqnarray*}
This contradicts $R(Z, W, \bar{Z}, \bar{W}) < 0$ and therefore proves the first statement
of the Proposition. For the
second statement, observe that for sufficiently
small $\epsilon$, say $\epsilon\le \frac{\delta}{4}k(p)$,
$\widetilde R=R-\e\I$ satisfies the weaker pinching condition.
\end{proof}

The consequence is the recent important result of Brendle and
Schoen.

\bigskip

\begin{corollary}[Brendle-Scheon]\label{bs2}
Assume that $(M, g_0)$ is a compact Riemannian manifold. Assume that
the sectional curvature of $g_0$ satisfies that
$$\frac{1+\delta}{4}k(p)\le R(X, Y, X, Y)/|X\wedge Y|^2\le
(1-\delta)k(p)
$$
for some continuous function $k(p)>0$ and constant $\delta>0$.
Then the normalized
Ricci flow deforms it into metric of constant curvature.
\end{corollary}

\bigskip

In the next section we shall prove a generalized version of this result.

\section{Generalization}

In this section we shall generalize Corollary \ref{bs2}. We first
start with the following proposition.

\begin{proposition} Let $(M, g(t))$ be a solution the Ricci flow. Assume that at $t=0$,
$R-f_0(x)\I$ has nonnegative complex sectional curvature for some
continuous $f\ge 0$. Let $f(x, t)$ be the solution to $\left(D_t
-\Delta \right)f(x, t)=0$ with the initial data $f(x,0)=f_0(x)$.
Then $\widetilde{R}=R-f(x,t)\I$ has nonnegative complex sectional
curvature for $t>0$.
\end{proposition}

\begin{proof}
By Lemma 2.1 of \cite{BW} is easy to check that
\begin{equation}
\label{eqn:identity}
Q(\widetilde R)=Q(R)-2f\Ric(R)\wedge \id +(n-1) f^2 \I
\end{equation}
where $\id$ is the identity of $\R^n=T_pM$. Since $f(x,t)$ satisfies $\left(\frac{\partial}{\partial t} -\Delta \right)f(x, t)=0$,
\begin{equation}
\label{eqn:heateqn}
\left(D_t-\Delta\right)\widetilde{R} = Q(R)
\end{equation}
By assumption $\widetilde{R}$ has nonnegative complex sectional curvature at $t=0$.
Therefore there is a first time $t_0$ (possibly at $t=0$) at which for some $Z, W \in T_pM \otimes \C$ we have
$\widetilde{R}(Z, W, \Zbar, \Wbar)=0$. By the maximum principle applied to (\ref{eqn:heateqn}) the theorem
follows if we can show $Q(R)(Z,W, \Zbar, \Wbar)\ge 0$.
Since $\widetilde{R}$ is an algebraic curvature operator we can
apply the results of Section 2 to conclude that  $Q(\widetilde R)(Z, W, \Zbar, \Wbar)\ge 0$. The result then
follows using (\ref{eqn:identity}) if we can show that
\begin{equation}
\label{eqn:inequ}
\left(2f\Ric(R)\wedge \id -(n-1) f^2 \I\right)(Z, W, \Zbar, \Wbar)\ge 0,
\end{equation}
To verify (\ref{eqn:inequ}) first notice that at $(p, t_0)$, $\widetilde R$ has nonnegative Ricci curvature.
Hence at $(p, t_0)$
$$
A\doteqdot \Ric(\widetilde{R})=\Ric(R)-(n-1)f \id \ge 0
$$
as element of $S^2(\R^n)$. From this, at $(p, t_0)$,
$$
\Ric(R) \ge (n-1)f \id \ge 0
$$
Thus, at $(p, t_0)$,
\begin{eqnarray*}
\left(2f\Ric(R)\wedge \id -(n-1) f^2 \I\right) = f\Ric(R) \wedge \id + fA\wedge \id.
\end{eqnarray*}
On the other hand
\begin{eqnarray*}
A\wedge \id (Z, W, \Zbar, \Wbar)&=&\frac{1}{2}\langle A(Z)\wedge
W+Z\wedge A(W), \Zbar\wedge \Wbar\rangle\\
&=& \frac{1}{2} \left(\langle A(Z), \Zbar\rangle |W|^2 +\langle
A(W), \Wbar\rangle |Z|^2\right)\\
&\, &  -\frac{1}{2}\left(\langle A(W), \Zbar \rangle \langle Z,
\Wbar \rangle +\langle A(Z), \Wbar \rangle \langle W, \Zbar \rangle
\right)
\end{eqnarray*}
which is nonnegative by the Cauchy-Schwartz inequality and $A \geq 0$.
Similarly,
$$
\Ric(R) \wedge \id(Z, W, \Zbar, \Wbar) \geq 0
$$
The result follows.
\end{proof}

\begin{corollary} If $f_0(x)\ge 0$ and $f_0(x_0)>0$ for some $x_0$,
then $R$ has positive complex sectional curvature for
$t>0$.
\end{corollary}
This together with Proposition  \ref{dualyz} implies the following
result.

\begin{corollary}
Let $(M, g_0)$ be a compact Riemannian manifold. Assume that there
exists continuous function $k(p), \delta(p) \ge 0$ such that
$\mathcal{P}=\{p\, | k(p)>0\}$ is dense and $\delta(p_0)>0$ for some
$p_0\in \mathcal{P}$, with the property that
$$
\frac{1+\delta(p)}{4}k(p)\le R(X, Y, X, Y)/|X\wedge Y|^2\le
(1-\delta(p))k(p).
$$
Then the normalized Ricci flow deforms $(M, g_0)$ into a metric of
constant curvature.
\end{corollary}

\section{Characterization of various invariant curvature cones}

In \cite{BS}, the authors introduced two invariant curvature cones $\tilde C$
and $\hat C$ motivated from their result that the nonnegativity of the isotropic
curvature is preserved under Ricci flow. Let $\pi: \R^n \times \R^2\to \R^n$ 
be the projection and define $\hat{R}(x, y, z, w)=R(\pi(x), \pi(y), \pi(z), \pi(w))$
where $x, y, z, w \in T( \R^n \times \R^2)$. Recall from \cite{BS} that
$$
\hat C=\{ R\, |\hat R \ \, \mbox{ has nonnegative isotropic
curvature}\}
$$
The $\tilde C$ cone is defined  similarly using $\R^n \times \R$.

 We shall show that in fact the  cone of nonnegative complex
sectional curvature as  used in this paper
is the same as $\hat C$.\footnote{After the circulation of an
earlier version of this paper, the authors were informed by Brendle
and Schoen of the following statement: $M$ has nonnegative complex sectional curvature is
equivalent to $M\times \R^4$ has nonnegative isotropic curvature. This
motivated the current section.} We are indebted to Nolan Wallach for
the following result.

\begin{proposition}
The following are equivalent:

(1) $R\in \hat C$;

(2) $R$ has non-negative complex sectional curvature.
\end{proposition}

\begin{proof}
It suffices to show that given $Z, W \in \C^n$ linearly independent,
there exist extensions $\widetilde Z=Z+ue_1+v e_2$ and
$\widetilde{W}=W+xe_1+ye_2$ of $Z,W$ to vectors  $\widetilde{Z} ,  \widetilde{W}$ 
in $\C^n \times \C^2$ such that
$\operatorname{Span}\{\widetilde{Z}, \widetilde{W}\}$ is an
isotropic plane. Here $\{e_1, e_2\}$ is an orthornormal basis of
the factor  $\R^2$ in the definition of $\hat C$. 

The existence of such an extension is
equivalent to the solution of the matrix equation
$
X X^t =A
$
with
$$
X=\left(\begin{matrix} u& v\cr x&y\end{matrix}\right),\quad \quad
A=\left(\begin{matrix} a & c\cr c& b\end{matrix}\right)
$$
where $a=-\langle Z, Z \rangle$, $b=-\langle W, W\rangle$,
$c=-\langle Z, W \rangle$. That the matrix equation can be solved follows from the fact that quadratic
form $as^2 +2cst +b t^2$ can be diagonalized by transformations of
$\operatorname{GL}(2, \C)$.
\end{proof}

The second result characterizes the $\tilde C$ cone.  Let $Z, W \in \C^n$ be linearly independent.
We say the $2$-vector $Z\wedge W$ is isotropic if:
$$0 = \langle Z\wedge W, Z\wedge W\rangle \doteqdot \langle Z\, Z\rangle \langle W,  W\rangle - \langle Z,  W\rangle^2.$$

\begin{proposition}  The following are equivalent:

(1) $R\in \tilde C$;

(2) $R$ is non-negative on any isotropic $2$-vector $Z\wedge W$.

\end{proposition}

\begin{proof}
To show (1) implies (2) we suppose that $Z, W \in \C^n$ are linearly independent and satisfy:
$$\langle Z\, Z\rangle \langle W,  W\rangle - \langle Z,  W\rangle^2 = 0.$$
Then there exist complex scalars $a, b$ such that $a^2=-\langle Z, Z \rangle$, $b^2=-\langle W, W\rangle$,
$ab=-\langle Z, W \rangle$. Hence $\widetilde Z=Z+a e_1$ and $\widetilde W=W+b e_1$ span an isotropic
$2$-plane, where $e_1$ is a unit vector in $\R$. Since $R$ has non-negative isotropic curvature on $\R^n \times \R$,
$$R(\widetilde Z, \widetilde W, \bar{ \widetilde Z}, \bar{\widetilde W}) \geq 0.$$ 
Hence,
$$R(Z, W, \bar{ Z}, \bar{W}) \geq 0.$$

The converse is similar and left to the reader.
\end{proof}

\bibliographystyle{amsalpha}

\end{document}